\documentclass[11pt]{amsart}
\usepackage[a4paper,top=2.5cm,bottom=2.5cm,left=2.5cm,right=2.5cm]{geometry}

\usepackage{amsmath}
\usepackage{amsthm}
\usepackage{amssymb}

\usepackage{breqn}

\usepackage{graphicx}
\usepackage[colorlinks=true, allcolors=blue]{hyperref}

\newtheorem{theorem}{Theorem}
\newtheorem{proposition}[theorem]{Proposition}
\newtheorem{lemma}[theorem]{Lemma}
\newtheorem{question}[theorem]{Question}
\newtheorem*{example}{Example}

\title{The Alexander polynomial of twisted torus knots}

\author{Kyungbae Park}
\address{Department of Mathematics, Kangwon National University}
\email{kyungbaepark@kangwon.ac.kr}

\author{Adnan}
\address{Department of Mathematics, Kangwon National University}
\email{adnanshahab35@kangwon.ac.kr}


\begin{document}
\begin{abstract}
Twisted torus knots are a generalization of torus knots, obtained by introducing additional full twists to adjacent strands of the torus knots. In this article, we present an explicit formula for the Alexander polynomial of twisted torus knots. Our approach utilizes a presentation of the knot group of twisted torus knots combined with Fox’s free differential calculus. As applications, we provide a lower bound for the genus of certain families of twisted torus knots and identify families of twisted torus knots that are not $L$-space knots.
\end{abstract}

\maketitle

\section{Introduction}
Given a quadruple of integers $(p,q;r,s)$ such that $\gcd(p,q)=1$, $1< r\leq |p|+|q|$ and $s\in\mathbb{Z}$, the twisted torus knot $T(p,q;r,s)$ is a generalization of torus knots, achieved by applying $s$ full twists to $r$ adjacent strands in the standard diagram of torus knot $T(p,q)$ on the flat torus. In particular, when $r < p$, which is the case we mainly consider in this paper, this knot can also be described as the closure of the braid corresponding to the braid word
\[
    (\sigma_1\sigma_2\dots\sigma_{p-1})^q(\sigma_1\dots\sigma_{r-1})^{rs}
\] 
with $p$ strands. Figure \ref{fig:TTK} illustrates an example of the twisted torus knot $T(7,4;3,2)$. Note that $T(p,q;r,s)$ with $r\leq q$ is isotopic to $T(q,p;r,s)$, and $T(p,-q;r,s)$ is isotopic to the mirror image of $T(p,q;r,-s)$. Considering our main interests in this paper, such as the Alexander polynomial and the genus of a knot, are invariant under the mirror imaging of a knot, we will assume that $p>q>0$ with $\gcd(p,q)=1$, $1<r<p$, and $s\in\mathbb{Z}$ for twisted torus knot $T(p,q;r,s)$ throughout this paper.

Twisted torus knots were first introduced by Dean in his study of Seifert fibered Dehn surgery \cite{Dean-1996, Dean-2003}. Since then, twisted torus knots have provided interesting examples of small hyperbolic knots \cite{CDW-1999, CKP-2004}, knots that admit unique Heegaard splitting of their complements \cite{Moriah-Sedgwick-2009}, and hyperbolic knots inducing certain types of hyperbolic structures \cite{Himeo-Teragaito-2023}, among other properties. They have also been studied in the context of the $L$-space conjecture as examples of $L$-spaces knots \cite{Vafaee-2015, CGHV-2016, Tran-2020}. 

The geometric types of twisted torus knots have been intensively studied by Lee, and also by Lee and Paiva. In particular, the classification of twisted torus knots that are the unknot was completely achieved in \cite{Lee-2014}. It is also known that some classes of twisted torus knots are, in fact, torus knots \cite{Lee-2019, Lee-2021, Lee-Paiva-2022}, although the complete classification of twisted torus knots that are torus knots has not yet been fully resolved.

However, not much is known about the polynomial knot invariants for twisted torus knots. Since some twisted torus knots are known to be the unknot, or more generally, torus knots as mentioned above, the Alexander polynomials of twisted torus knots in those families are implicitly known. In \cite{Morton-2006}, Morton provided an explicit formula for the Alexander polynomial of twisted torus knots of the type $T(p,q;2,s)$. His main technique involved the Burau representation alongside an inductive argument based on the skein relation of the Alexander polynomial. If $r=2$, the other skein pairs of a twisted torus knot are also twisted torus knots or links, allowing the use of a recursive formula to obtain the Alexander polynomial of the twisted knot. However, this technique does not work for twisted torus knots with $r>2$, as the other pairs in the skein relation are no longer twisted torus knots or links. The knot Floer homology, which refines Alexander polynomial invariant, for twisted torus knots with $p=qk\pm 1$ is studied in \cite{Vafaee-2015}. In \cite{BD-2023}, the Jones polynomial of twisted torus knots of the type $T(p,q;2,s)$ is computed. 

\begin{figure}[!t]
    \centering
    \includegraphics[angle=90, width=0.7\textwidth]{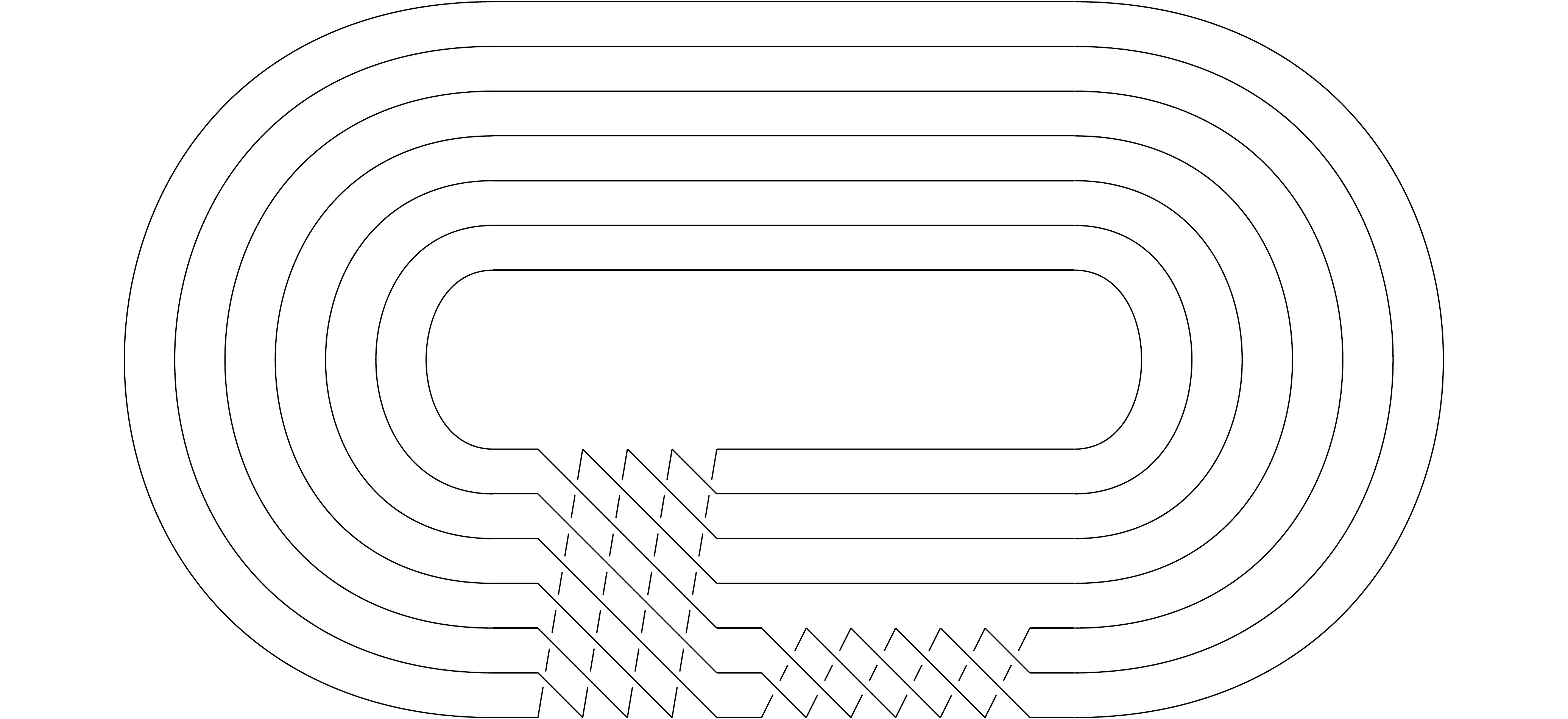}
    \caption{Twisted torus knot $T(7,4;3,2)$}
    \label{fig:TTK}
\end{figure}

Our main result in this paper is to provide an explicit formula for any twisted torus knot. To achieve this, instead of the skein relation, we employ a classical technique dating back to the 1950s, due to Fox \cite{Fox-1953}, known as Fox's free differential calculus, or Fox calculus for short. We first obtain a presentation of the knot group of twisted torus knots (see Theorem \ref{thm:fund}) by placing the knot on a genus two surface and using Seifert-Van Kampen's theorem. Subsequently, we apply Fox calculus to determine the Alexander polynomial of twisted torus knots.

Given a twisted torus knot $T(p,q;r,s)$, let $m, \overline{k'}, Q',\text{ and }\overline{k_i} \text{ and } Q_i \text{ for }i=1,\dots,r-1$ be nonnegative integers determined by considering modular arithmetic relationships involving $p,q,\text{ and }r$ as defined at the beginning of Section \ref{sec:fund}. Then:
\begin{theorem}\label{thm:Alex}
    Let $p$ and $q$ be relatively prime integers with $p>q>0$, and let $1< r<p$ and $s\in\mathbb{Z}$. The Alexander polynomial of the twisted torus knot $T(p,q;r,s)$ is given, up to multiplication by $\pm t^{\pm n}$, as follows:
    \[ \Delta_{T(p,q;r,s)}(t)=\frac{1-t}{(1-t^p)(1-t^q)(1-t^{r})}\left(\tilde{X}(t)Y(t)-X(t)\tilde{Y}(t)\right),\]
    where
    \begin{align*}
       X(t)&=1-(1-t^{rs})\sum_{i=1}^{r-1}t^{\overline{k_i}p+(i-1)rs}-t^{pq+(r-1)rs},\\
       \tilde{X}(t)&=\begin{cases}
            1-t^{\overline{k'}p},&\text{if } m=0,\\
            1-(1-t^{rs})\displaystyle\sum_{i=1}^{m}t^{\overline{k_i}p+(i-1)rs}-t^{\overline{k'}p+mrs},&\text{otherwise,}
       \end{cases}\\
       Y(t)&=1-(1-t^{rs})\sum_{i=1}^{r-1}t^{Q_iq+(i-1)rs}-t^{pq+(r-1)rs},\\
       \tilde{Y}(t)&=\begin{cases}
           1-t^{Q'q},&\text{if } m = 0,\\
            1-(1-t^{rs})\displaystyle\sum_{i=1}^{m}t^{Q_iq+(i-1)rs}-t^{Q'q+mrs},&\text{otherwise.}       \end{cases}
   \end{align*}
\end{theorem}

Utilizing the fact that half of the \emph{degree} of the Alexander polynomial of a knot (defined as the difference between the highest power and the lowest power of the polynomial) provides a lower bound for the genus of the knot (defined as the minimal genus of a Seifert surface for the knot), we can deduce information about the genus of $T(p,q;r,s)$ in terms of $p,q,r, \text{ and } s$ for certain families.
\begin{theorem}\label{thm:genus}
    Let $g(T(p,q;r,s))$ denote the genus of the twisted torus knot $T(p,q;r,s)$, where $\gcd(p,q) = 1$, $p > q > 0$, $1 < r < p$, and $s \in \mathbb{Z}$.
    
    If $s>0$ or $r|s| < q$, then
    \[\frac{(p-1)(q-1) + (r-1)rs}{2}\leq g(T(p,q;r,s)) \leq \frac{(p-1)(q-1) + (r-1)r|s|}{2}.\]
    
    In particular, if $s > 0$, then
    \[g(T(p,q;r,s)) = \frac{(p-1)(q-1) + (r-1)rs}{2}.\]
\end{theorem}
\begin{proof}
    If $s > 0$ or $r|s| < q$, then the degree of the Alexander polynomial of $T(p,q;r,s)$ is given by $(p-1)(q-1) + (r-1)rs$ by Proposition \ref{prop:degree1} and Proposition \ref{prop:degree2}. Therefore, 
    \[\frac{(p-1)(q-1) + (r-1)rs}{2}\leq g(T(p,q;r,s)).\]
    Let $\Sigma$ be the Seifert surface obtained by applying the Seifert algorithm to the knot diagram depicted by the closure of the braid, as shown in Figure \ref{fig:TTK} for $T(p,q;r,s)$. Considering that there are $p$ Seifert circuits and $(p-1)q+(r-1)r|s|$ crossings in the Seifert algorithm, we have 
    \[
        g(\Sigma)=\frac{(p-1)(q-1)-(r-1)r|s|}{2}.
    \] 
    Since the genus of $T(p,q;r,s)$ is bounded above by the genus of $\Sigma$, the theorem follows.
\end{proof}
Note that the genus of $T(p,q;r,s)$ for $s>0$ can also be easily determined using a result by Stalling \cite{Stallings-1978}, since this knot is the closure of a positive braid. If $s < 0$, however, it presents a more interesting case in the study of twisted torus knots, making the determination of the genus for general twisted torus knots more challenging. The assumption $r|s|<q$ can be interpreted as a scenario in which the additional negative twists added to the torus knot are relatively fewer than the original twists of the torus knot. Conversely, we can also obtain a lower bound for the genus using the degree in the opposite case when $r|s|$ is relatively greater (see Proposition \ref{prop:degree3}).

In particular, we note that the lower bound for the genus given in the theorem above is sharp for certain families of twisted torus knots with $s<0$. For example, the genus of $T(p, p-1; 3, -1)$ is given by \[g(T(p, p-1; 3, -1))=\frac{(p-1)(p-2)}{2}-3,\]
which coincides with half of the degree of its Alexander polynomial. See the last part of Section \ref{sec:Degree}.

One interesting topic in low-dimensional topology is the classification of all knots in $S^3$ that produce lens spaces through integral surgery along the knots, known as the \emph{Berge conjecture}. In \cite{Morton-2006}, Morton showed the existence of an infinite family of twisted torus knots of the type $T(p,q;2,s)$ whose Alexander polynomials have coefficients that are neither $\pm1$ nor $0$. Consequently, by the result of Ozsv\'ath and Szab\'o's \cite[Corollary 1.3]{Ozsvath-Szabo-2005}, these knots do not admit lens space surgeries; in other words, any integral surgery on these knots does not produce lens spaces. In \cite{Vafaee-2015}, Vafaee classified twisted torus knots of the type $T(qk\pm1,q;r,s)$ that admit $L$-space surgery, a generalization of lens space surgery in the context of Heegaard Floer theory. Using our formula, one can identify infinite families of twisted torus knots that do not admit lens space surgery through a simple computation. The following is one such new example.
\begin{theorem}\label{thm:lens_space_surgery}
    For each $n\geq 0$ and $s\geq 2$, the Alexander polynomial of the twisted torus knot $T(9+2n, 7+2n; 3, s)$ has a coefficient with an absolute value of at least $2$. Consequently, this knot does not admit lens space surgery (and, more generally, is not an $L$-space knot).
\end{theorem}

\subsection*{Further Discussion}
\subsubsection*{Classification of twisted torus knot} It is interesting to explore how the Alexander polynomial serves as a powerful invariant for twisted torus knots. It is not difficult to observe that there are many pairs of twisted torus knots with different parameters that are isotopic, even in the case when $s>0$. For example, $T(1+3n, 3;2, 1+m)$ and $T(3+2m, 2;3, n)$ are isotopic for each $n>0$ and $m>0$. We implemented Python code to compute the Alexander polynomial of twisted torus knots using our formula (see the second author's homepage). Specifically, let $\mathcal{T}$ denote the set of $7,450$ twisted torus knots with $3\leq p\leq 20$, $0<q<p$ with $\gcd(p,q)=1$, $1< r<p$ and $1\leq s\leq 5$. We found that each pair in $\mathcal{T}$ with the same Alexander polynomial also has the same volume (up to an error $<10^{-8}$ with SnapPy's results) if they are hyperbolic knots (determined by SnapPy \cite{SnapPy}). Conversely, any pair of two hyperbolic twisted torus knots in $\mathcal{T}$ with the identical volume (determined by SnapPy) also has the same Alexander polynomial. These observations lead us to pose the following question.
\begin{question}
    Is Alexander polynomial a complete invariant for positively twisted torus knots?
\end{question}

Here ``positively twisted" refers to the case when $s>0$. The question above is not expected to hold for negatively twisted torus knots, as we identified pairs of negatively twisted torus knots that share the same Alexander polynomial but have different volumes according to SnapPy. For example, $T(4, 3; 2, -5)$ and $T(6, 5; 4, -2)$ form such a pair.

\subsubsection*{Geometric types of twisted torus knots} Since the Alexander polynomial of torus knots is well-known, and the Alexander polynomial of satellite knots can be expressed in terms of those of their companion and pattern, the Alexander polynomial of knots can serve as an obstruction for a knot to be classified as a specific torus knot or a particular satellite knot. By combining topological arguments, the Alexander polynomial can demonstrate that a given twisted torus knot is not of the presumed type. We plan to explore this topic further in future work.

\subsubsection*{Knot Floer homology} Knot Floer homology \cite{Rasmussen-2003, Ozavath-Szabo-2004}, a categorification of the Alexander polynomial, is a powerful knot invariant that detects the exact genus of knots, their fiberedness, and other properties. The Knot Floer homology of twisted torus knots has only been computed for certain families in \cite{Vafaee-2015}. Given that our approach utilizes diagrams of twisted torus knots on a Heegaard surface, and that the Alexander polynomial provides the Euler characteristic of Knot Floer homology, this raises intriguing questions.
\begin{question}
    What is the knot Floer homology of twisted torus knots?
\end{question}

\subsubsection*{Generalization to other families of knots} Twisted torus knots are originally defined for the parameter $2\leq r \leq |p|+|q|$ (see \cite[Chapter 3.1]{Dean-2003}). Although we restrict our attention to the case where $2\leq r<\max\{|p|,|q|\}$, similar arguments can be applied to the general case, though they would yield a different type of formula. 

Twisted torus knots can also be generalized or may overlap with larger families of knots and links, such as \emph{twisted torus links}, links obtained by concatenating two braid words for torus links, and \emph{$T$-links} (equivalently \emph{Lorenz links}), which are links formed by concatenating multiple braid words for torus links $T(p_i, q_i)$ with $q_i>0$ \cite{Birman-Kofman-2009}. We anticipate that the techniques used in this work can also be extended to these families to compute their Alexander polynomials.
\begin{question}
    What is the Alexander polynomial of twisted torus links or $T$-links?
\end{question}

\subsection*{Acknowledgements}
The authors thank the referee for valuable suggestions that improved the manuscript. They also thank Thiago de Paiva for his interest and helpful discussions. Kyungbae Park was supported by the National Research Foundation of Korea (NRF) grant funded by the Korea government (No. RS-2022-NR073368).

\section{Fundamental Group of Twisted Torus Knots}\label{sec:fund}
In this section, we utilize the standard genus two Heegaard decomposition of $S^3$ along with Seifert-Van Kampen's theorem to provide a presentation of the knot group of the twisted torus knot $T(p,q;r,s)$. To accurately describe this presentation, it is necessary to define some notations relating to the residue arithmetic of $p$, $q$ and $r$. For the remainder of our discussion, we will fix the tuple $(p,q;r,s)$ under the conditions that $p>q>0$ with $\gcd(p,q)=1$, $1< r<p$, and $s\in\mathbb{Z}$.

Let $[x]$ denote the residue of $x$ modulo $p$, i.e. the unique integer $[x]$ such that $0 \leq [x] < p$ and $[x]\equiv x$ modulo $p$. Define $Q$ as the set of residues of the multiples $nq^{-1}$, $n=0,\dots,r-1$ of 
$q^{-1}$ modulo $p$, where $q^{-1}$ denotes the multiplicative inverse of $q$ modulo $p$. The elements of $Q$ are ordered in ascending sequences as $Q_0,\dots,Q_{r-1}$. This set can be expressed as:
\begin{equation*}
    Q:=\{[nq^{-1}]|n=0,\dots,r-1\}=\{Q_0<Q_1<\dots<Q_{r-1}\}.
\end{equation*}
Define 
    \[d_i:=Q_i-Q_{i-1} \text{ for } i=1,\dots,r-1,\]
and
    \[d_r:=p-Q_{r-1}.\] 
Let $R$ be the set of residues of multiples $-nq^{-1}$ for $n=1,\dots,q$ of $-q^{-1}$ modulo $p$. This can be expressed as:
    \[R:=\{[-nq^{-1}]|n=1,\dots,q\}\]
Define $k_i$ as the cardinality of the set of elements $l$ in $R$ such that $Q_{i-1}\leq l<Q_i$ for $i=1,\dots,r-1$, and $k_r$ as the cardinality of the set of elements $l$ in $R$ satisfying $Q_{r-1}\leq l$.
    \[k_i:=\left|\{l\in R|Q_{i-1}\leq l<Q_i\}\right|\text{ for } i=1,\dots,r-1,\]
and
    \[k_r:= \left|\{l\in R|Q_{r-1}\leq l\}\right|.\]
Define $Q':=[rq^{-1}]$. Note that $Q'$ is not contained in $Q$ since we assume $p$ and $q$ are relatively prime. Consequently, we have the ordering: 
    \[Q_0<\dots<Q_m<Q'<Q_{m+1}<\dots<Q_{r-1}\] 
for some $0\leq m\leq r-1$. Denote $d':=Q'-Q_m$ and $k':=|\{l\in R|Q_{m}\leq l< Q'\}|$.
We further define \[\overline{k_i}:=k_1+\dots+k_i\text{ for }i=1\dots r-1,\]
and
\[\overline{k'}:=k_1+\dots+k_m+k'.\]

\begin{figure}[!t]
    \centering   
    \includegraphics[width=0.8\textwidth]{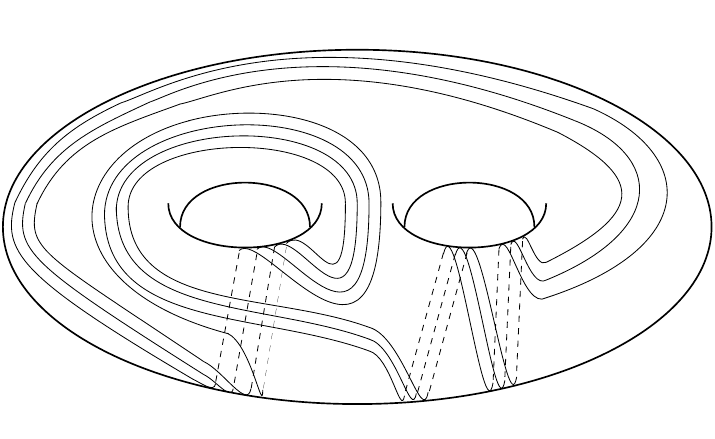}
    \caption{Twisted torus knot $T(7,4;3,2)$ on a genus two surface}
    \label{fig:TTK_on_surface}
\end{figure}

\begin{example}\rm We give an example illustrating the computation of the above numbers for an infinite family of triples $(p,q,r)$, which will be used in Section~\ref{sec:L-space}. Let $(p,q,r)=(9+2\ell,7+2\ell,3)$ where $\ell$ is a nonnegative integer. Note that $[q^{-1}]=[(7+2\ell)^{-1}]=4+\ell$. Therefore \[
    Q=\left\{\left[n(4+\ell)\right]|n=0,1,2\right\}=\{0, 4+\ell, 8+2\ell \},\]
so we have
    \[Q_0=0, \quad Q_1 = 4+\ell,\quad \text{and}\quad Q_2=2(4+\ell),\]
and 
    \[d_1=d_2= 4+\ell\quad\text{and}\quad d_3=p-Q_2= 1.\]
We also have 
    \[R=\left\{[-n(4+\ell)]| n = 1,2, \dots,(7+2\ell)\right\}=\{1,2,\dots,3+\ell, 5+\ell,6+\ell\dots,8+2\ell\},\]
and hence 
    \[k_1 = k_2 = 3+\ell, \quad\text{and}\quad k_3=1.\] 
Thus 
    \[\overline{k_1} =3+\ell \quad\text{and}\quad \overline{k_2}=2(3+\ell).\]
Since $Q'=[3(4+\ell)]=3+\ell$ and so $Q_0<Q'<Q_1$, we obtain \[
    m=0\quad\text{and}\quad\overline{k'}=2+\ell.\]

\end{example}

\begin{theorem}\label{thm:fund}
         Let $p$ and $q$ be relatively prime with $0<q<p$, and $2\leq r<p$ and $s\in\mathbb{Z}$. Then the knot group of the twisted torus knot $T(p,q;r,s)$ has the following presentation.
	   \[\pi_1(S^3- T(p,q;r,s))\cong\langle x,y,z,w|r_1,       r_2, r_3\rangle,\]
         where	
        \[	\begin{cases}
		r_1=x^{k_1}y^sx^{k_2}y^s\cdots x^{k_m}y^sx^{k'}y\cdot z^{d'}wz^{d_m}wz^{d_{m-1}}\cdots wz^{d_1}\\
		r_2=y^sw\\
		r_3=x^{k_1}y^sx^{k_2}y^s\cdots x^{k_r}y^s\cdot wz^{d_r}wz^{d_{r-1}}\cdots wz^{d_1}
	\end{cases}\]
\end{theorem}
\begin{proof}
First, observe that a twisted torus knot $K=T(p,q;r,s)$ is a \emph{double torus knot}, i.e., it can be placed on a genus two surface $\Sigma$ standardly located in the 3-sphere $S^3$, as depicted in Figure \ref{fig:TTK_on_surface}. Consider the decomposition of $S^3$ into two genus two handlebodies $H_1$ and $H_2$ along $\Sigma$. Let $\nu(K)$ be the small tubular neighborhood of $K$. Denote $H_1'$, $H_2'$, and $\Sigma'$ as the closure of $H_1-\nu(K)$, $H_2-\nu(K)$, and $\Sigma-\nu(K)$, respectively. The knot complement $S^3-\nu(K)$ can then be decomposed into $H_1'$ and $H_2'$, with $H_1'\cap H_2'=\Sigma'$. 
\[
    S^3-\nu(K)=H_1'\cup_{\Sigma'} H_2'
\]

It is straightforward to deduce that $H_1'$ and $H_2'$ are homeomorphic to genus two handlebodies, and $\Sigma'$ is homeomorphic to a torus with two punctures, considering the Euler characteristic of $\Sigma'$ and the fact that $K$ is a nonseparating knot on $\Sigma$. Note that the fundamental groups of $H_1'$ and $H_2'$ are free groups generated by two generators. We denote these as $\pi_1(H_1')\cong\langle x,y\rangle$ and $\pi_1(H_2')\cong\langle z,w\rangle$, where the generating loops are illustrated in Figure \ref{fig:generators}. Meanwhile, the fundamental group of $\Sigma'$ is a free group generated by three generators. We denote the generating loops of $\Sigma'$ by $a$, $b$, and $c$ and will describe them in more detail later. Applying Seifert-Van Kampen's theorem, the knot group of $K$ can be presented as follows:
\[\pi_1(S^3-\nu(K))=\langle x,y,z,w | i(a)j(a)^{-1},i(b)j(b)^{-1},i(c)j(c)^{-1}\rangle,\]
where $i$ and $j$ are the induced homomorphisms by the inclusion maps $\Sigma'\hookrightarrow H_1'$ and $\Sigma'\hookrightarrow H_2'$, respectively.
\begin{figure}[t!]
    \centering   
    \includegraphics[width=\textwidth]{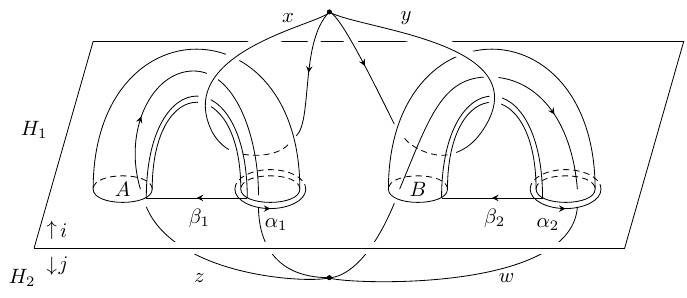}
    \caption{Genus two Heegaard decomposition of $S^3$ and generators of the fundamental group of each handlebody}
    \label{fig:generators}
\end{figure}

Now, we consider the generating loops of $\pi_1(\Sigma')$. Place the knot $K$ on the Heegaard surface as depicted in Figure \ref{fig:knot_on_Heegaard}, where the pairs of disks ($A$, \scalebox{-1}[1]{$A$}) and ($B$, \scalebox{-1}[1]{$B$}) represent the attaching regions of two $1$-handles. For a representation of Heegaard splitting in $\mathbb{R}^2$, see  \cite[Chapter 4.3]{Gompf-Stipsicz-1999}. The $p$ strands of $K$ pass through the boundary of the 1-handle with the attaching regions ($A$, \scalebox{-1}[1]{$A$}), and the $r$ strands of $K$ pass through the boundary of the 1-handle with the attaching region ($B$, \scalebox{-1}[1]{$B$}). Choose a base point on the outer part of $K$ in the diagram. The generator $a$ of $\pi_1(\Sigma')$ is represented by a loop starting at the base point, traveling into the boundary of the right attaching region, $\partial
\scalebox{-1}[1]{A}$, and then directly returning to the base point. Note that the loop must travel while avoiding any contact with $K$. The generator $b$ of $\pi_1(\Sigma')$ is represented by a loop starting at the base point, traveling into the boundary of the right attaching region, $\partial\scalebox{-1}[1]{B}$, and then returning to the base point. The generator $c$ is represented by a parallel copy on the left side of the knot $K$ on $\Sigma'$. The orientation of $c$ is depicted in Figure \ref{fig:knot_on_Heegaard}. It can be verified that these loops generate $\pi_1(\Sigma')$ as they do $H_1(\Sigma';\mathbb{Z})$.

\begin{figure}[t!]
    \centering
    \includegraphics[width=1\textwidth]{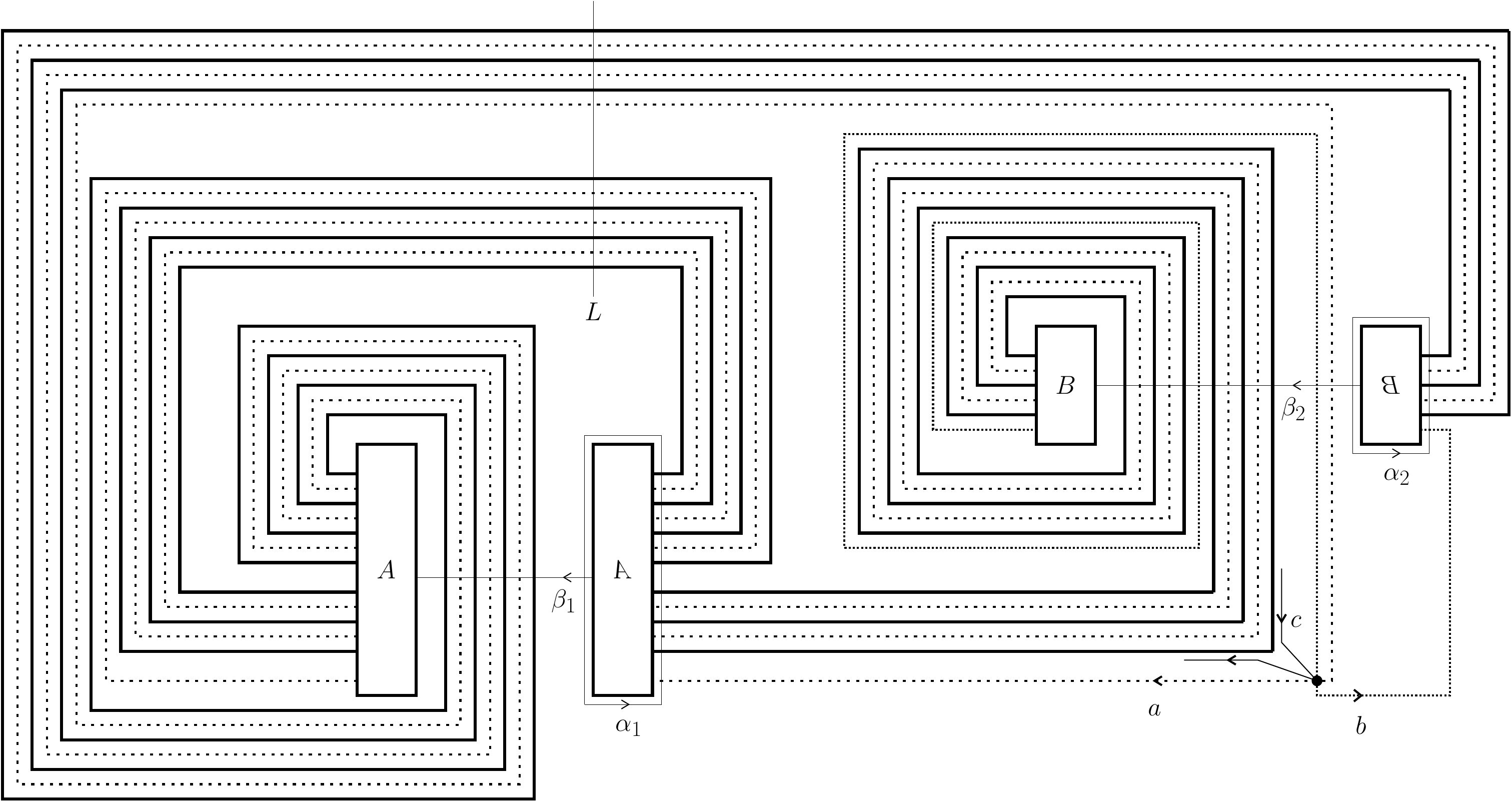}
    \caption{Twisted torus knot $T(7,4;3,2)$ on a Heegaard surface, along with loops $a$, $b$ and $c$ that generate $\pi_1$ of the complement of the knot}
    \label{fig:knot_on_Heegaard}
\end{figure}

Now, we discuss how one can determine the values of the maps $i$ and $j$ at a loop. Note that the generators $x$ and $y$ of $\pi_1(H_1')$ are homotopic to the loops $\alpha_1$ and $\alpha_2$ on $\Sigma$, respectively, while the generators $z$ and $w$ of $\pi_1(H_2')$ are homotopic to the loops $\beta_1$ and $\beta_2$, respectively (see Figure \ref{fig:generators}). To compute the image of a loop under the map $i$, we follow the loop from the base point and observe how it sequentially winds around the attaching regions $A$ and $B$. Alternately, this observation corresponds to tracing how the loop intersects with $\beta_1$ and $\beta_2$. Similarly, the image under $j$ can be determined by tracing how the loop sequentially intersects with $\alpha_1$ and $\alpha_2$.

For instance, consider the loop $b$ which intersects $\beta_2$ $|s|$ times while traveling from the base point and returning to the base point. It is also important to note that the orientation of the intersection with $\beta_2$ matches that of $\alpha_2$ if $s>0$, and is opposite if $s<0$. Therefore, 
\[i(b)=y^s\]
Whereas, $b$ intersects $\alpha_2$ once during its travel, and the orientation of this intersection is opposite to that of $\beta_2$. Hence,
\[j(b)=w^{-1}\]

Now, let us consider $j(c)$. First, note that as traveling while traveling along $c$ from the base
point and returning to the base point, it moves out from $\partial A$ exactly $p$ times. At each time, there are following four cases until it exits again from $A$.
\begin{itemize}
    \item It intersects $\beta_1$ once, then approaches $B$ side, intersects $\beta_2$ $|s|$ times, and then exits from $\partial A$ again.
    \item It directly approaches $B$ side, intersects $\beta_2$ $|s|$ times, and then exits again from $\partial A$ again.
    \item It intersects $\beta_1$ once, and exits from $\partial A$ again.
    \item It exits from $\partial A$ again without intersecting any $\beta_1$ or $\beta_2$.
\end{itemize}
This case is determined by which sector of $\partial A$ $c$ comes out from and which sector of the line segment $L$ $c$ passes through. We label sectors of $\partial A$ and $L$ as depicted in Figure \ref{Sector of dA and L}. Let $t_n$ be the label of the sector of $\partial A$ where $c$ comes out at the $n$th visit to $\partial A$, where let $u_n$ be the label of the sector of $L$ where $c$ passes through at the $n$th visit to 
$L$. It is easy to verify that $t_n=\left[(p-1)-(n-1)q\right]_p$ and $u_n=[q+(n-1)q]_p=[nq]_p$, where $[k]_p$ denotes the residue of $k$ modulo $p$. Then, the word $\xi_n$ in $\pi_1(H_1')$ corresponding from the $n$th to the $(n+1)$th visit of $c$ to $\partial A$ is given by the following.
\begin{equation*}
    \xi_n=
    \begin{cases}
        xy^s&\text{if }t_n < q\text{ and }u_n<r\\
        y^s&\text{if }t_n \geq q\text{ and }u_n< r\\
        x&\text{if }t_n < q\text{ and } u_n\geq r\\
        1&\text{if }t_n \geq q\text{ and } u_n\geq r\\
    \end{cases}
\end{equation*}
Then \[i(c)=\prod_{n=1}^p\xi_n,\] and this corresponds to
\[i(c)=x^{k_1}y^sx^{k_2}y^s\cdots x^{k_r}y^s\]
which is easily followed by the definition of $k_i$. Note that the set $Q$ corresponds to the set of orders of visits of $c$ to $L$ when the label at $L$ is less than $r$, and $k_i$ counts the number of times of visits between $Q_{i-1}$ and $Q_{i}$ when the label at $\partial A$ is $\leq q$. The word for $j(c)$ can be computed similarly to have
\[j(c)=\prod_{n=1}^p{\eta_n},\]
where
\begin{equation*}
    \eta_n=
    \begin{cases}
        z^{-1}w^{-1}&\text{if}\quad u_n<r\\
        z^{-1}&\text{if}\quad u_n\geq r\\
    \end{cases}
\end{equation*}
and hence
\[j(c)=z^{-d_1}w^{-1}\cdots z^{-d_{r-1}}w^{-1}z^{-d_r}w^{-1},\]
by considering the definition of $d_i$.
\begin{figure}[t!]
    \centering
    \includegraphics[width=0.5\textwidth]{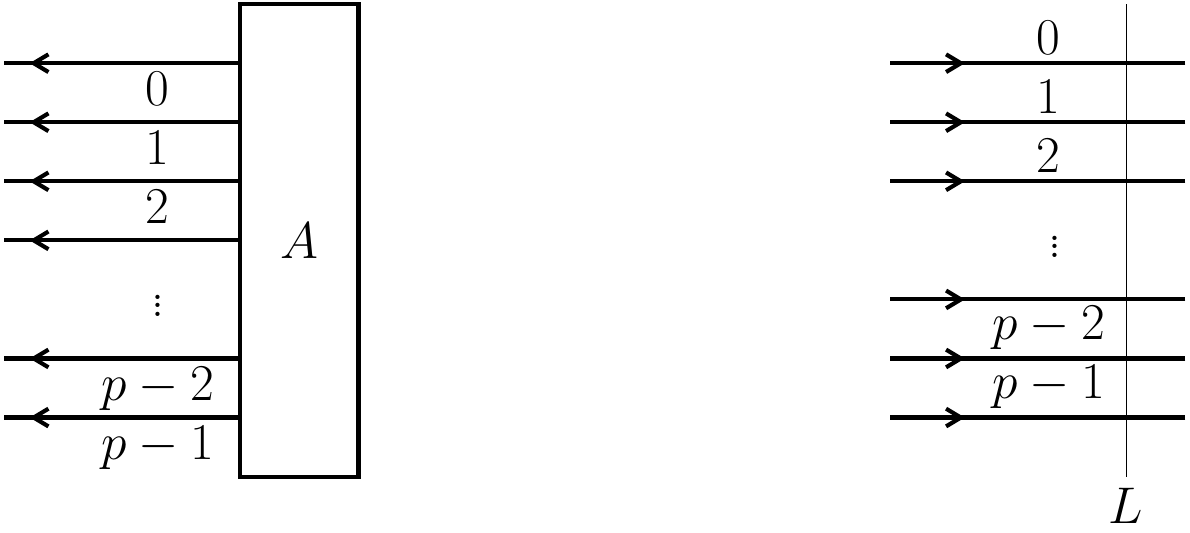}
    \caption{Sectors of $\partial A$  and $L$}
    \label{Sector of dA and L}
\end{figure}

The same argument applies to obtaining $i(a)$ and $j(a)$. However, in this case, the loop $a$ meets $\partial A$ exactly $Q'$ times since $Q':=[rq^{-1}]$ is the solution $n$ of the equation $u_n=r$ and $a$ returns to the base point when it passes through the sector of $L$ labeled $r$. Considering this situation, we have
\[i(a)=\prod_{n=1}^{Q'} \xi_n\cdot y.\]
Hence,
\[i(a)=x^{k_1}y^sx^{k_2}y^s\cdots x^{k_m}y^sx^{k'}y.\]
Similarly, we obtain
\[j(a)=\prod_{n=1}^{Q'} \eta_n.\]
Finally, we have
\[j(a)=z^{-d_1}w^{-1}\cdots z^{-d_{m-1}}w^{-1}z^{-d_m}w^{-1}z^{-d'}\]
Then, our theorem follows from the Seifert-van Kampen theorem.\end{proof}

\section{Alexander Polynomial of Twisted Torus Knots}
\subsection{Fox's free calculus derivative and Alexander polynomial}
Although we have a presentation of the knot group for twisted torus knots, it is generally not straightforward to determine if two given presentations are isomorphic. In \cite{Fox-1953}, Fox introduced a method to derive the Alexander polynomial from a presentation of the knot group. We summarize the procedure here. Refer to Chapter 11 of \cite{Lickorish-1997} for the details.

Let $F$ be the free group of $n$ generators $\langle x_1,x_2,\dots,x_n\rangle$ and $\mathbb{Z}[F]$ be the integral group ring of $F$. The derivative in Fox's free calculus with respect to $x_i$ in $\mathbb{Z}[F]$ is a map $\frac{\partial}{\partial x_i}: F\to \mathbb{Z}[F] $ defined by 
\begin{align} 
    \frac{\partial x_j}{\partial x_i}&=\delta_{i,j},\text{ and}\\
    \frac{\partial (xy)}{\partial x_i}&=\frac{\partial x}{\partial x_i}+x\frac{\partial y}{\partial x_i} \quad\text{for } x,y\in F, \label{fox; def}
\end{align} 
where $\delta_{i,j}$ is the Kronecker delta function. These derivative maps can be uniquely extended to $\mathbb{Z}[F]$. The following identities can be easily deduced for the above.
\begin{lemma}\label{lem:Fox_derivatives}
    Let $x\in F$. We have the following identities of Fox derivatives:
  \begin{align}
      \frac{\partial 1}{\partial x_i}&=0,\label{eq; 3}\\ 
      \frac{\partial x^{-1}}{\partial x_i}&=-x^{-1}\frac{\partial x}{\partial x_i} \label{eq; 4}\\
      \frac{\partial x^n}{\partial x_i}&= \frac{1-x^n}{1-x}\frac{\partial x}{\partial x_i}, \label{eq; 5}
  \end{align}
where
\[
    \frac{1-x^n}{1-x}=
    \begin{cases}
        0 &\text{if }n=0,\\
        \displaystyle\sum_{i=0}^{n-1} x^i &\text{if }n>0,\\
        -\displaystyle\sum_{i=n}^{-1}x^i  &\text{if }n<0.
    \end{cases}
\]
\end{lemma}
\begin{proof}
The identity (\ref{eq; 3}) directly follows from (\ref{fox; def}) by setting $x=1$ and $y=1$. The identity (\ref{eq; 4}) can be obtained from the following equation:
\begin{equation*}
    0=\frac{\partial 1}{\partial x_i}=\frac{\partial (x^{-1}x)}{\partial x_i}=\frac{\partial x^{-1}}{\partial x_i}+x^{-1}\frac{\partial x}{\partial x_i} 
\end{equation*}
The identity (\ref{eq; 5}) can be obtained by induction on $|n|$. For $n=0,1,$ and $-1$, it is clear that the property (\ref{eq; 5}) is satisfied. Assuming the induction hypothesis holds for some $n$, then by the identity (\ref{fox; def}), we have
\begin{equation*}
    \frac{\partial x^{n+1}}{\partial x_i}=\frac{\partial (x^nx)}{\partial x_i}=\frac{\partial x^n}{\partial x_i}+x^n\frac{\partial x}{\partial x_i} = {\displaystyle \sum_{i=1}^{n-1}} x^i\frac{\partial x}{\partial x_i} +x^n\frac{\partial x}{\partial x_i} = \frac{1-x^{n+1}}{1-x} \frac{\partial x}{\partial x_i}
\end{equation*}

Similarly, if $n<-1$,
\begin{equation*}
  \frac{\partial x^{n-1}}{\partial x_i}=\frac{\partial (x^nx^{-1})}{\partial x_i}= \frac{\partial x^n}{\partial x_i}+x^n\frac{\partial x^{-1}}{\partial x_i} = -{\displaystyle \sum_{i=n}^{-1}} x^i\frac{\partial x}{\partial x_i} -x^{n-1}\frac{\partial x}{\partial x_i} = \frac{1-x^{n-1}}{1-x}\frac{\partial x}{\partial x_i}
\end{equation*}
\end{proof}

Let $G=\langle x_1,x_2,\dots, x_n | r_1,r_2,\dots,r_{n-1}\rangle$ be a presentation of the knot group of a knot $K$. Let $\phi: \mathbb{Z}[F]\to \mathbb{Z}[G]$ be the ring homomorphism induced from the canonical map $F\to G$. Consider the following Jacobian matrix $J$ of order $(n-1)\times n$:
\[J=\phi\left(\frac{\partial r_i}{\partial x_j}\right).\]
Let $\alpha\colon \mathbb{Z}[G]\to H_1(S^3-K;\mathbb{Z})\cong \mathbb{Z}\langle t\rangle$ be the map induced from the abelianization map from $G$ to $H_1(S^3-K;Z)$, and let $\Tilde{\alpha}\colon \mathbb{Z}[G]\to \mathbb{Z}[t,t^{-1}]$ be the induced map from $\alpha$. We also denote $\tilde{\alpha}=\alpha$. The image of $J$ under the abelianization map $\alpha (J)=\alpha\phi\left(\frac{\partial r_i}{\partial x_j}\right)$ is called the \textit{Alexander matrix} of $G$.
 
Let $m_1,m_2,\dots,m_{n-1}$ be the $(n-1)\times (n-1)$ minors of $\alpha(J)$. The ideal generated by  $m_1,m_2,\dots,m_{n-1}$ is known to be a principal ideal, and its generator is the \emph{Alexander polynomial} $\Delta_K(t)$ of $K$, up to multiplication by units $\pm t^{ k}$ for some $k\in \mathbb{Z}$.

\subsection{Alexander polynomial of twisted torus knots}
In this section, we compute the Alexander polynomial of twisted torus knots by applying Fox calculus to the presentation of the knot group described in Section \ref{sec:fund}.

Recall the presentation of the knot group of $K=T(p,q;r,s)$ from Theorem \ref{thm:fund}:
\[\pi_1(S^3-K)=\langle x, y, z, w|r_1, r_2, r_3\rangle.\] 
We also carry all notation introduced in Section \ref{sec:fund}. Before proceeding further, note the following identities, which will be useful later.
\begin{lemma}\label{lem:identity}
    For each $i=0,\dots,r-1$, if $Q_i=[nq^{-1}]$ for some $n\in\{0,\dots, r-1\}$, then
    \[
        Q_iq=\overline{k_i}p+n,
    \] 
    and 
    \[
        Q'q=\overline{k'}p+r.
    \]
\end{lemma}
\begin{proof}
    We provide a proof of the last identity, as the others follow in a similar manner. Recall that $Q'=[rq^{-1}]$ and $\overline{k'}$ denote the number of elements in \[
    \{[-\ell q^{-1}]|\ell=1,\dots,q,\, [-\ell q^{-1}]<Q'\}.\]
    Observe that $Q'$ is the minimal positive integer $m$ such that $mq\equiv r$ modulo $p$. Hence, there exists a positive integer $t$ such that $Q'q=tp+r$. We claim that $t=\overline{k'}$. For each $1\leq v\leq t$, we have $m_vq$, which is the greatest multiple of $q$ less than $vp$. Therefore, since $1\leq vp-m_vq\leq q$, $t$ equals the number of solutions of $x$ such that $0<x<Q'$ and $x=[-\ell q^{-1}]$ for some $1\leq \ell\leq q$, which corresponds to $\overline{k'}$.

\end{proof}

By taking Fox's derivatives to $r_1, r_2,$ and $r_3$ with respect to $x$,$y$,$z$, and $w$ using Lemma \ref{lem:Fox_derivatives}, we obtain
\begin{align*}
    \frac{\partial r_1}{\partial x}&=\frac{1-x^{k_1}}{1-x}+x^{k_1}y^s\frac{1-x^{k_2}}{1-x}+\dots+x^{k_1}y^s\dots x^{k_{m-1}}y^s\frac{1-x^{k_m}}{1-x}+x^{k_1}y^s\dots x^{k_{m}}y^s\frac{1-x^{k'}}{1-x},\\
    \frac{\partial r_1}{\partial y}&=\left(x^{k_1}+x^{k_1}y^s+\dots+x^{k_1}y^s\dots x^{k_m}\right)\frac{1-y^s}{1-y}+x^{k_1}y^s\dots x^{k_m}y^sx^{k'},\\
    \frac{\partial r_1}{\partial z}&=i(a)\left(\frac{1-z^{d'}}{1-z}+z^{d'}w\frac{1-z^{d_m}}{1-z}+z^{d'}wz^{d_m}w\frac{1-z^{d_{m-1}}}{1-z}+\dots+z^{d'}wz^{d_m}w\dots w\frac{1-z^{d_1}}{1-z}\right),\\
    \frac{\partial r_1}{\partial w}&=i(a)\left(1+wz^{d_m}+wz^{d_m}wz^{d_{m-1}}+\dots+wz^{d_m}\cdots wz^{d_2}\right),\\
\end{align*}
\begin{equation*}
    \frac{\partial r_2}{\partial x}=0,\quad
    \frac{\partial r_2}{\partial y}=\frac{1-y^s}{1-y},\quad
    \frac{\partial r_2}{\partial z}=0,\quad
    \frac{\partial r_2}{\partial w}=y^s,
\end{equation*}
\begin{align*}
    \frac{\partial r_3}{\partial x}&=\frac{1-x^{k_1}}{1-x}+x^{k_1}y^s\frac{1-x^{k_2}}{1-x}+\dots+x^{k_1}y^s\cdots x^{k_{r-1}}y^s\frac{1-x^{k_r}}{1-x},\\
    \frac{\partial r_3}{\partial y}&=\left(x^{k_1}+x^{k_1}y^sx^{k_2}+\dots+x^{k_1}y^s\cdots x^{k_{r-1}}y^sx^{k_{r}}\right)\frac{1-y^s}{1-y},\\
    \frac{\partial r_3}{\partial z}&=i(c)\left(w\frac{1-z^{d
    _r}}{1-z}+wz^{d_r}w\frac{1-z^{d_{r-1}}}{1-z}+\dots+wz^{d_r}wz^{d_{r-1}}w\cdots z^{d_2}w\frac{1-z^{d_1}}{1-z}\right),\text{ and}\\
    \frac{\partial r_3}{\partial w}&=i(c)\left(1+wz^{d_{r}}+wz^{d_{r}}wz^{d_{r-1}}+\dots+wz^{d_{r}}\cdots wz^{d_2}\right),
\end{align*}
where $i(a)$ and $i(c)$ are words computed in the proof of Theorem \ref{thm:fund}.

Note that $\alpha(x)=t^p$, $\alpha(y)=t^r$, $\alpha(z)=t^{-q}$, $\alpha(w)=t^{-rs}$ by taking a generator $t$ of $H_1(S^3-K;\mathbb{Z})$ represented by a meridian of $K$ and considering how we choose $x$, $y$, $z$, and $w$ in Figure \ref{fig:generators}. By applying the map $\alpha$ to each derivative, we have:
\begin{align*}
    \alpha\frac{\partial r_1}{\partial x}&=\frac{1}{1-t^p}\left\{1-(1-t^{rs})\sum_{i=1}^m t^{\bar{k_i}p+(i-1)rs}-t^{\bar{k'}p+mrs}\right\}=:\frac{\tilde{X}}{1-t^p},\\
    \alpha\frac{\partial r_1}{\partial y}&=\frac{1-t^{rs}}{1-t^r}\sum_{i=1}^m t^{\overline{k_i}p+(i-1)rs}+t^{\overline{k'}p+mrs}=\frac{1-\tilde{X}-t^{\overline{k'}p+mrs}}{1-t^r}+t^{\overline{k'}p+mrs},\\
    \alpha\frac{\partial r_1}{\partial z}&=\frac{t^{\overline{k'}p+(1-Q')q+r}}{1-t^q}\left\{1-(1-t^{rs})\sum_{i=1}^{m}t^{Q_iq+(i-1)rs}-t^{Q'q+mrs}\right\}=:\frac{t^q}{1-t^q}\tilde{Y},\text{ and}\\
    \alpha\frac{\partial r_1}{\partial w}&=t^{rs}\sum_{i=1}^{m}t^{Q_{i}q+(i-1)rs}=\frac{t^{rs}}{1-t^{rs}}\left(1-\tilde{Y}-t^{Q'q+mrs}\right).
\end{align*}
where we applied Lemma \ref{lem:identity} for the third identity. Similarly, we have
\begin{align*}
    \alpha\frac{\partial r_2}{\partial x}=\alpha\frac{\partial r_2}{\partial z}=0,\quad\alpha\frac{\partial r_2}{\partial y}=\frac{1-t^{rs}}{1-t^r},\quad\text{and}\quad\alpha\frac{\partial r_2}{\partial w}=t^{rs}, 
\end{align*}
and
\begin{align*}
    \alpha\frac{\partial r_3}{\partial x}&=\frac{1}{1-t^p}\left\{1-(1-t^{rs})\sum_{i=1}^{r-1}t^{\overline{k_i}p+(i-1)rs}-t^{pq+(r-1)rs}\right\}=\frac{1}{1-t^p}X,\\
    \alpha\frac{\partial r_3}{\partial y}&=\frac{1-t^{rs}}{1-t^r}\sum_{i=1}^{r} t^{\overline{k_i}p+(i-1)rs}=\frac{1}{1-t^r}(1-X-t^{pq+r^2s}),\\
    \alpha\frac{\partial r_3}{\partial z}&=\frac{t^{q}}{1-t^q}\left\{1-(1-t^{rs})\sum_{i=1}^{r-1}t^{Q_iq+(i-1)rs}-t^{pq+(r-1)rs}\right\}=\frac{t^{q}}{1-t^q}Y,\text{ and}\\
    \alpha\frac{\partial r_3}{\partial w}&=\sum_{i=1}^{r}t^{Q_iq+irs}=\frac{t^{rs}}{1-t^{rs}}\left(1-Y-t^{pq+r^2s}\right),
\end{align*}
where we define $Q_r=p$. 

Now, consider the following Alexander matrix of $\pi_1(S^3-K)$.
\renewcommand{\arraystretch}{2.5}
\[\begin{pmatrix}
    \alpha\frac{\partial r_1}{\partial x} & \alpha\frac{\partial r_1}{\partial y} & \alpha\frac{\partial r_1}{\partial z} & \alpha\frac{\partial r_1}{\partial w} \\
    0 & \alpha\frac{\partial r_2}{\partial y} & 0 & \alpha\frac{\partial r_2}{\partial w} \\
    \alpha\frac{\partial r_3}{\partial x} & \alpha\frac{\partial r_3}{\partial y} & \alpha\frac{\partial r_3}{\partial z} & \alpha\frac{\partial r_3}{\partial w}
\end{pmatrix}\]
Then, the minors $m_1$, $m_2$, $m_3$ and $m_4$ of the above will be the following: 
\begin{align*}
    m_1&=\alpha\frac{\partial r_1}{\partial z}\left(\alpha\frac{\partial r_2}{\partial y}\alpha\frac{\partial r_3}{\partial w}- \alpha\frac{\partial r_2}{\partial w}\alpha\frac{\partial r_3}{\partial y}\right)+\alpha\frac{\partial r_3}{\partial z} \left(\alpha\frac{\partial r_1}{\partial y}\alpha\frac{\partial r_2}{\partial w}-\alpha\frac{\partial r_1}{\partial w}\alpha\frac{\partial r_2}{\partial y}\right)\\
    m_2&=\alpha\frac{\partial r_2}{\partial w}\left(\alpha\frac{\partial r_1}{\partial x}\alpha\frac{\partial r_3}{\partial z}-\alpha\frac{\partial r_1}{\partial z}\alpha\frac{\partial r_3}{\partial x}\right)\\
    m_3&=\alpha\frac{\partial r_2}{\partial y}\left(\alpha\frac{\partial r_1}{\partial x}\alpha\frac{\partial r_3}{\partial w}- \alpha\frac{\partial r_1}{\partial w}\alpha\frac{\partial r_3}{\partial x}\right)-\alpha\frac{\partial r_2}{\partial w} \left(\alpha\frac{\partial r_1}{\partial x}\alpha\frac{\partial r_3}{\partial y}-\alpha\frac{\partial r_1}{\partial y}\alpha\frac{\partial r_3}{\partial x}\right)\\
    m_4&=\alpha\frac{\partial r_2}{\partial y}\left(\alpha\frac{\partial r_1}{\partial x}\alpha\frac{\partial r_3}{\partial z}-\alpha\frac{\partial r_1}{\partial z}\alpha\frac{\partial r_3}{\partial x}\right)
\end{align*}
It is straightforward to observe that $t^{rs}(1-t^r)m_4=(1-t^{rs})m_2$. We further claim that similar identities exist for $m_1$ and $m_3$, each relating to $m_2$. 

\begin{lemma}\label{lem:minors}
    $(1-t^r)m_1=-(1-t^p)m_2$ and $t^q(1-t^r)m_3=-(1-t^q)m_2$.
\end{lemma}
\begin{proof}
Note that
\begin{equation}\label{eq:m_2}
    m_2=\frac{t^{q+rs}}{(1-t^p)(1-t^q)}\left(\tilde{X}Y-\tilde{Y}X\right)
\end{equation}
and
\begin{align*}
    m_1&=\frac{t^q}{1-t^q}\tilde{Y}\left\{\frac{t^{rs}}{1-t^r}\left(1-Y-t^{pq+r^2s}\right)-\frac{t^{rs}}{1-t^r}(1-X-t^{pq+r^2s})\right\}\\
    &\quad + \frac{t^{q}}{1-t^q}Y\left\{t^{rs}\left(\frac{1-\tilde{X}-t^{\overline{k'}p+mrs}}{1-t^r}+t^{\overline{k'}p+mrs}\right)-\frac{t^{rs}}{1-t^r}\left(1-\tilde{Y}-t^{Q'q+mrs}\right)\right\}\\
    &=\frac{t^{q+rs}}{(1-t^q)(1-t^r)}(\tilde{Y}X-\tilde{X}Y).
\end{align*}
Note that we applied Lemma \ref{lem:identity} for the last equality. Hence, we get the identity \[
(1-t^r)m_1=-(1-t^p)m_2.\] 
Similarly, we can verify that
\begin{equation*}
    m_3=\frac{t^{rs}}{(1-t^p)(1-t^r)}(X\tilde{Y}-Y\tilde{X}),
\end{equation*}
and so $t^q(1-t^r)m_3=-(1-t^q)m_2$.
\end{proof}

Now, we prove Theorem \ref{thm:Alex}.
\begin{proof}[Proof of Theorem \ref{thm:Alex}]
Let $\displaystyle D(t) := \frac{(1-t)}{(1-t^r)}m_2$. By Lemma \ref{lem:minors}, we have
\begin{equation*}
    \frac{m_1}{D(t)} = \frac{1-t^p}{1-t}, \quad
    \frac{m_2}{D(t)} = \frac{1-t^r}{1-t}, \quad
    \frac{m_3}{D(t)} = \frac{1-t^q}{1-t}, \quad\text{and}\quad \frac{m_4}{D(t)} = \frac{1-t^{rs}}{1-t}.
\end{equation*}
Since $p$ and $q$ are relatively prime integers, it follows that
\[\gcd\left(\frac{1-t^p}{1-t}, \frac{1-t^q}{1-t}\right) = 1.\]
Therefore, we deduce that
\[\gcd\left(\frac{m_1}{D(t)}, \frac{m_2}{D(t)}, \frac{m_3}{D(t)}, \frac{m_4}{D(t)}\right) = 1.\]
Hence, the Alexander polynomial of $T(p,q;r,s)$, denoted by $\Delta_{T(p,q;r,s)}(t)$, is given by
\[\Delta_{T(p,q;r,s)}(t) = D(t) = \frac{1-t}{(1-t^p)(1-t^q)(1-t^r)}\left(\tilde{X}Y - X\tilde{Y}\right),\]
by Equation (\ref{eq:m_2}).
\end{proof}

\section{Degree of the Alexander polynomial of Twisted Torus Knots}\label{sec:Degree}
In this section, we discuss the degree of the Alexander polynomial of twisted torus knots. For positive values of $s$, the degree of Alexander polynomial is easily determined by the tuple $(p,q;r,s)$.
\begin{proposition}\label{prop:degree1}
    If $s>0$, 
    \[
        \deg(\Delta_{T(p,q;r,s)})=(p-1)(q-1)+(r-1)rs.
    \]
\end{proposition}
\begin{proof}
We consider the following Alexander polynomial
\[
    \Delta_{T(p,q;r,s)}(t) = \frac{1-t}{(1-t^p)(1-t^q)(1-t^{r})}(\tilde{X}Y - X\tilde{Y}).
\]
in Theorem \ref{thm:Alex}. If $s > 0$, then the highest powers in $\tilde{X}Y$ and $X\tilde{Y}$ are $\overline{k'}p + mrs + pq + (r-1)rs$ and $Q'q + mrs + pq + (r-1)rs$, respectively. By Lemma \ref{lem:identity}, it is clear that the highest power in $\tilde{X}Y - X\tilde{Y}$ is $Q'q + mrs + pq + (r-1)rs$.

Clearly,
\[
    X = \tilde{X} - A \quad \text{and} \quad Y = \tilde{Y} - B,
\]
where 
    \[A =  \ t^{\overline{k'}p + mrs} + (1-t^{rs})\sum_{i=m+1}^{r-1}t^{\overline{k_i}p + (i-1)rs} - t^{pq + (r-1)rs},\] 
and
    \[B =  \ t^{Q'q + mrs} + (1-t^{rs})\sum_{i=m+1}^{r-1}t^{Q_iq + (i-1)rs} - t^{pq + (r-1)rs}.\]
Then we have,
\begin{equation}\label{AB}
    \tilde{X}Y - X\tilde{Y} = \tilde{X}B - \tilde{Y}A.
\end{equation}
If $s > 0$, then the lowest powers of $t$ in $\tilde{X}B$ and $A\tilde{Y}$ are $Q'q + mrs$ and $\overline{k'}p + mrs$, respectively. By Lemma \ref{lem:identity}, it is clear that the lowest power of $t$ in $\tilde{X}Y - X\tilde{Y}$ is $\overline{k'}p + mrs$. Therefore,
\[
    \deg(\Delta_{T(p,q;r,s)}) = (p-1)(q-1) + (r-1)rs.
\]
\end{proof}
However, if $s<0$, describing the degree of the Alexander polynomial of $T(p,q;r,s)$ in terms of $p,q,r,\text{ and }s$ becomes challenging due to the difficulty in determining the highest of the lowest powers of $t$ in each $X$, $\tilde{X}$, $Y$, and $\tilde{Y}$. Nevertheless, we can determine it for the following special cases.
\begin{proposition}\label{prop:degree2}
    Suppose $s<0$ and $|rs|<q<p$. Then 
    \begin{equation*}
        \deg(\Delta_{T(p,q;r,s)})=(p-1)(q-1)+(r-1)rs. 
    \end{equation*}
\end{proposition}
\begin{proof}
    Given that $|rs| < q < p$, it follows that $|rs| < d
    _iq$ and $|rs| < {k_i}p$ for all $i$. Consequently, the highest and the lowest powers of $t$ in $\tilde{X}Y - X\tilde{Y}$ remain unchanged from the case when $s > 0$.
\end{proof}

The situation described above can be viewed as the case where the number of additional negative twists is relatively fewer than the original positive twists of the torus knot. The following presents the opposite scenario.
\begin{proposition}\label{prop:degree3}
    Suppose $s<0$, and $|rs|>d_iq$ and $|rs|>k_i p$ for all $i=1,\dots,r-1$. Then 
    \begin{enumerate}
        \item $\deg\left(\Delta_{T(p,q;r,s)}\right)=(k_1-\overline{k_{r-1}}-1)p+(Q'-1)q-(r-1)(rs+1)$ if $m=0$.
        \item $\deg\left(\Delta_{T(p,q;r,s)}\right)=(p-1)(q-1)+(Q_1-Q'-Q_{r-1})q-(r-1)rs$ if $m=r-1$.
        \item $\deg\left(\Delta_{T(p,q;r,s)}\right)=(\overline{k_{m+1}}-\overline{k_{r-1}}-1)p+(Q_1-Q_m-1)q-(r-1)(rs+1)$ if $0<m<r-1$.
    \end{enumerate}
\end{proposition}
\begin{proof}
    If $m=0$, then $\tilde{X}=1-t^{\overline{k'}p}$ and $\tilde{Y}=1-t^{Q'q}$. Since $|rs|>d_iq$ and $|rs|>k_ip$ for all $i$, the highest powers of $t$ in $\tilde{X}Y$ and $X\tilde{Y}$ are $Q_1q+\overline{k'}p$ and $Q'q+k_1p$ respectively. By Lemma \ref{lem:identity} the highest power in $\tilde{X}Y-X\tilde{Y}$ is $Q'q+k_1p$. Similarly the lowest powers in $\tilde{X}Y$ and $X\tilde{Y}$ are $Q_{r-1}q+(r-1)rs$ and $\overline{k_{r-1}}p+(r-1)rs$ respectively. 
    Again by Lemma~\ref{lem:identity}, the lowest power in $\tilde{X}Y-X\tilde{Y}$ is $\overline{k_{r-1}}p+(r-1)rs$. Hence, considering the degree of the formula in Theorem \ref{thm:Alex},
    \[ \text{deg}\left(\Delta_{T(p,q;r,s)}\right)=(k_1-\overline{k_{r-1}}-1)p+(Q'-1)q-(r-1)(rs+1).\]

    If $m=r-1$, we have:
        \[\tilde{X}=1-(1-t^{rs})\sum_{i=1}^{r-1}t^{\overline{k_i}p+(i-1)rs}-t^{\overline{k'}p+(r-1)rs}\] 
    and
        \[\tilde{Y}=1-(1-t^{rs})\sum_{i=1}^{r-1}t^{Q_iq+(i-1)rs}-t^{Q'q+(r-1)rs}.\]
    We write $X= \tilde{X}- A$ and $Y=\tilde{Y}-B$, where 
    \[
        A=t^{pq+(r-1)rs}-t^{\overline{k'}p+(r-1)rs}\quad\text{and}\quad
        B=t^{pq+(r-1)rs}-t^{Q'q+(r-1)rs}.
    \]
    Then $\tilde{X}Y-X\tilde{Y} = \tilde{X}B-\tilde{Y}A$. The highest powers in $\tilde{X}B$ and $\tilde{Y}A$ are $pq+k_1p+(r-1)rs$ and  $pq+Q_1q+(r-1)rs$, respectively. Similarly, the lowest powers in $\tilde{X}B$ and $\tilde{Y}A$ are $Q'q+\overline{k_{r-1}}p+2(r-1)rs$ and $Q_{r-1}q+\overline{k'}p+2(r-1)rs$, respectively. By Lemma \ref{lem:identity}, the highest and the lowest powers in $\tilde{X}Y-X\tilde{Y}$ are $pq+Q_1q+(r-1)rs$ and $\overline{k'}p+Q_{r-1}q+2(r-1)rs$, respectively. Hence, \[
        \text{deg}\left(\Delta_{T(p,q;r,s)}\right)=(p-1)(q-1)+(Q_1-Q'-Q_{r-1})q-(r-1)rs.\]
        
    Now, consider the case when $0<m<r-1$. The highest powers of $t$ in $\tilde{X}B$ and $\tilde{Y}A$ of Equation (\ref{AB}) are $Q_{m+1}q+k_1p+mrs$ and $Q_1q+\overline{k_{m+1}}p+mrs$ respectively, and the lowest powers in $\tilde{X}B$ and $\tilde{Y}A$ are $Q_{r-1}q+\overline{k_m}p+(r+m-1)rs$ and $Q_mq+\overline{k_{r-1}}p+(r+m-1)rs$ respectively. 
    
    We claim that $Q_1q-k_1p > Q_{m+1}q-\overline{k_{m+1}}p$. Let $n_1:=Q_1q-k_1p$ and $n_{m+1}:=Q_{m+1}q-\overline{k_{m+1}}p$. Then, by Lemma~\ref{lem:identity}, we have $Q_1=[n_1q^{-1}]$ and $Q_{m+1}=[n_{m+1}q^{-1}]$. Consider the following sequence of residues of multiples of $q^{-1}$ modulo $p$: 
        \[a_1=[q^{-1}], a_2=[2q^{-1}], \dots, a_r=[rq^{-1}].\]
    Note that $[rq^{-1}]$ cannot be either the minimum or the maximum value in this sequence, as that would imply $m=0\text{ or }r-1$, by the definition of $m$, contradicting our assumption. Hence, $n_1$ corresponds to the index of the minimal value, and $n_{m+1}$ corresponds to the index of the next greater value than $[rq^{-1}]$. Moreover, since $m\neq r-1$, we must have $n_1\neq 1$. Examining the sequence of residues modulo $p$, we observe that the minimal value appears later in the sequence than the next greater value following $[rq^{-1}]$. This implies that $n_1>n_{m+1}$, proving the claim.

    Similarly, we can show that $Q_{m}q-\overline{k_{m}}p < Q_{r-1}q-\overline{k_{r-1}}p$. Then, it follows from these inequalities and Equation~(\ref{AB}) that the highest and lowest powers in $\tilde{X}Y-X\tilde{Y}$ are $Q_1q+\overline{k_{m+1}}p+mrs$ and $Q_mq+\overline{k_{r-1}}p+(r+m-1)rs$, respectively. Therefore, we obtain
    \[\deg(\Delta_{T(p,q;r,s)})=(\overline{k_{m+1}}-\overline{k_{r-1}}-1)p+(Q_1-Q_m-1)q-(r-1)(rs+1).\]
\end{proof}

\begin{figure}[t!]
    \centering
    \includegraphics[width=\textwidth]{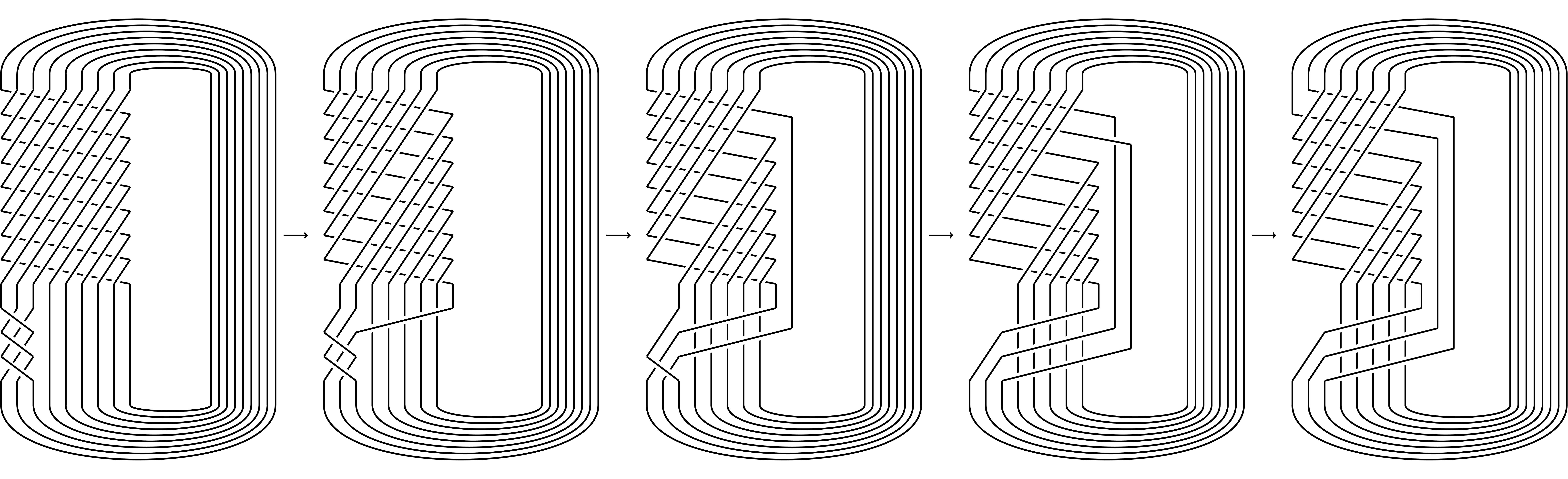}
    \caption{An isotopy of $T(9,8;3,-1)$ to a diagram that yields a Seifert surface of minimal genus via Seifert algorithm.}
    \label{Genus of T(9,8;3,-1)}
\end{figure}

The degree of twisted torus knot $T(p,p-1;3,-1)$ provides a sharp lower bound for the genus of the knot. By applying a similar isotopy as shown in Figure \ref{Genus of T(9,8;3,-1)} and using the Seifert algorithm on the resulting diagram, we obtain a Seifert surface of $T(p,p-1;3,-1)$ with a genus that coincides with half of the degree of the Alexander polynomial of the knot. Hence, the genus of the knot is equal to the genus of the surface. Since the final diagram of the knot is, in fact, the closure of a positive braid, one can also show that the Seifert surface obtained from this diagram yields the genus of the knot.

\section{Twisted Torus Knots Not Admitting Lens Space Surgery}\label{sec:L-space}
Although our formula for the Alexander polynomial of a twisted torus knot involves integers related to modular arithmetic, it can be simplified by focusing on specific families of twisted torus knots. We present an example here to illustrate this simplification and to prove Theorem \ref{thm:lens_space_surgery} concerning lens space surgery.

Consider a twisted torus knot of the type $T(9+2n,7+2n;3,s)$. Recalling the computation from the example preceding Theorem \ref{thm:fund}, we have
\begin{align*}
    X(t)=&1-(1-t^{3s})\sum_{i=1}^{2}t^{i(3+n)(9+2n)+(i-1)3s}-t^{(9+2n)(7+2n)+6s},\\
    \tilde{X}(t)=&1-t^{(2+n)(9+2n)},\\
    Y(t)=&1-(1-t^{3s})\sum_{i=1}^{2}t^{i(4+n)(7+2n)+(i-1)3s}-t^{(9+2n)(7+2n)+6s},\quad\text{and}\\
    \tilde{Y}(t)=&1-t^{(3+n)(7+2n)}.
\end{align*}
The Alexander polynomial of $T(9+2n,7+2n;3,s)$ as follows:
\[\Delta_{T(9+2n,7+2n;3,s)}(t) = \frac{(1-t)(\tilde{X}(t)Y(t)-X(t)\tilde{Y}(t))}{(1-t^{9+2n})(1-t^{7+2n})(1-t^{3})}\]

Now, assume $s\geq 2$. We claim that the coefficient of $t^{31+15n+2n^2}$, which is the eighth lowest power term of $\Delta_{T(9+2n,7+2n;3,s)}(t)$, is $-2$. To determine this coefficient of $t^{31+15n+2n^2}$, we may disregard all terms in $\tilde{X}(t)Y(t)-X(t)\tilde{Y}(t)$ whose power is higher than $31+15n+2n^2$. Thus, under the assumptions for $n$ and $s$, $\tilde{X}(t)Y(t)-X(t)\tilde{Y}(t)$ only includes the following terms with powers less than or equal to this value:\[
    -t^{18+13n+2n^2}+t^{21+13n+2n^2}+t^{(3+n)(9+2n)}-t^{(4+n)(7+2n)}\]
Since \[
    \frac{1}{1-t^m}=1+t^m+t^{2m}+\dots\]
and considering the expansion of the five factors \[
    (\tilde{X}(t)Y(t)-X(t)\tilde{Y}(t))(1-t)(1+t^{(9+2n)}+t^{2(9+2n)}+\dots)(1+t^{7+2n}+t^{2(7+2n)}+\dots)(1+t^{3}+t^{6}+\dots)
\] of the Alexander polynomial, all terms with the power $31+15n+2n^2$ are as follows:\[
     t^{(18+13n+2n^2)+1+(9+2n)+3},\quad-t^{(18+13n+2n^2)+(7+2n)+6},\quad t^{(21+13n+2n^2)+(7+2n)+3},\]\[ 
     -t^{(21+13n+2n^2)+1+(9+2n)},\quad -t^{(3+n)(9+2n)+1+3},\quad -t^{(4+n)(7+2n)+3}.\]
Therefore, the coefficient of $t^{31+15n+2n^2}$ is $-2$, and Theorem \ref{thm:lens_space_surgery} follows.

\clearpage
\bibliographystyle{alpha}
\bibliography{references-ttk}

\end{document}